\documentclass[fleqn, preprint
]{zzz}
\usepackage{mathrsfs}
\usepackage[pdftex]{hyperref}
\hypersetup{
	colorlinks = true,
	urlcolor = blue,
	linkcolor = red,
	citecolor = black
}
\makeatletter
\def\eqnarray{\stepcounter{equation}\let\@currentlabel=\theequation
	\global\@eqnswtrue
	\tabskip\@centering\let\\=\@eqncr
	$$\halign to \displaywidth\bgroup\hfil\global\@eqcnt\z@
	$\displaystyle\tabskip\z@{##}$&\global\@eqcnt\@ne
	\hfil$\displaystyle{{}##{}}$\hfil
	&\global\@eqcnt\tw@ $\displaystyle{##}$\hfil
	\tabskip\@centering&\llap{##}\tabskip\z@\cr}
\def\endeqnarray{\@@eqncr\egroup
	\global\advance\c@equation\m@ne$$\global\@ignoretrue}
\def\@yeqncr{\@ifnextchar [{\@xeqncr}{\@xeqncr[5pt]}}
\makeatother
\newcommand\boro[1]{{\textcolor{black}{#1}}}
\newcommand{\yoro}[1]{{\textcolor{black}{#1}}}
\newcommand{\zoro}[1]{{\textcolor{black}{#1}}}
\usepackage[normalem]{ulem}
\usepackage{xparse}

\newcommand{\qhyp}[5]{\,\mbox{}_{#1}\phi_{#2}\!\left(
\genfrac{}{}{0pt}{}{#3}{#4};#5\right)}

\NewDocumentCommand{\qfrac}{smm}{%
 \dfrac{\IfBooleanT{#1}{\vphantom{\big|}}#2}{\mathstrut #3}%
}

\newtheorem{thm}[theorem]{Theorem}
\newtheorem{cor}[corollary]{Corollary}
\newtheorem{rem}[remark]{Remark}

\newtheorem{defn}[definition]{Definition}

\makeatletter
\def\eqnarray{\stepcounter{equation}\let\@currentlabel=\theequation
\global\@eqnswtrue
\tabskip\@centering\let\\=\@eqncr
$$\halign to \displaywidth\bgroup\hfil\global\@eqcnt\z@
$\displaystyle\tabskip\z@{##}$&\global\@eqcnt\@ne
\hfil$\displaystyle{{}##{}}$\hfil
&\global\@eqcnt\tw@ $\displaystyle{##}$\hfil
\tabskip\@centering&\llap{##}\tabskip\z@\cr}

\def\endeqnarray{\@@eqncr\egroup
\global\advance\c@equation\m@ne$$\global\@ignoretrue}

\def\@yeqncr{\@ifnextchar [{\@xeqncr}{\@xeqncr[5pt]}}
\makeatother

\newcommand{\CC}{{\mathbb C}}
\newcommand{\expe}{{\mathrm e}}
\newcommand{\CCast}{{\mathbb C}^\ast}

\newcommand{\SSS}{{\mathcal S}}

\begin{document}

\renewcommand{\PaperNumber}{***}

\FirstPageHeading

\ShortArticleName{Terminating Basic Hyper. Representations \& Transformations 
for the Askey--Wilson Polyn.}

\ArticleName{Terminating Basic Hypergeometric Representations and Transformations 
for the Askey--Wilson Polynomials}

\Author{Howard S. Cohl $^{1,}$*$^{,\dagger}$, Roberto S. Costas-Santos $^{2,\dagger}$ and 
Linus Ge $^{3,\dagger}$}

\AuthorNameForHeading{H. S. Cohl, R. S. Costas-Santos, and L. Ge}
\Address{$^\dag$ Applied and Computational Mathematics Division,
National Institute of Standards and Technology,
Gaithersburg, MD 20899-8910, USA
\URLaddressD{http://www.nist.gov/itl/math/msg/howard-s-cohl.cfm}
} 
\EmailD{howard.cohl@nist.gov} 

\Address{$^\S$ Dpto. de F\'isica y Matem\'{a}ticas,
Universidad de Alcal\'{a},
c.p. 28871, Alcal\'{a} de Henares, Spain
} 
\URLaddressD{http://www.rscosan.com}
\EmailD{rscosa@gmail.com} 

\Address{$^\ddag$ Department of Mathematics, University of Rochester, 
Rochester, NY 14627, USA}
\EmailD{lge6@u.rochester.edu} 


\ArticleDates{Received ?? \today in final form ????; Published online ????}

\abstract{In this survey paper, we exhaustively explore the terminating 
basic hypergeometric representations of the Askey--Wilson polynomials
and the corresponding terminating basic hypergeometric transformations 
that these polynomials satisfy.}


\section{Introduction}
\label{Introduction}
This paper is a study in $q$-calculus (typically taken with $|q|<1$).
The $q$-calculus (introduced by such luminaries as Leonhard Euler,
Eduard Heine and Garl Gustav Jacobi)
is a calculus of finite differences which becomes the standard infinitesimal calculus (introduced by
Isaac Newton and Gottfried Wilhelm Leibniz) in the limit
as $q\to 1$.
The work contained in this paper is
directly connected to 
properties of the Askey--Wilson polynomials $p_n(x;{\bf a}|q)$ (\cite[\S 14.1]{Koekoeketal}) 
which are at the very top of the $q$-Askey 
scheme (see e.g., \cite[Chapter 14]{Koekoeketal}). 
The Askey--Wilson polynomials
are basic hypergeometric orthogonal
polynomials with interpretations
in quantum group theory, 
combinatorics, and probability.
The applications of Askey--Wilson
polynomials include invariants of
links, ${\textit{3}}$-manifolds 
and ${\textit{6}}j$-symbols
(see e.g., \cite{Rosengren2007}).
The definition of the Askey--Wilson
polynomials in terms of terminating
basic hypergeometric series
are given in
Theorem \ref{AWthm} below. The Askey--Wilson polynomials are symmetric with respect to their 
four free parameters, that is, they remain unchanged upon 
interchange of any two of the
four free parameters. 
\boro{It should be emphasized that since
1970, the subjects of special functions
and special families of
orthogonal polynomials have gone
through major developments, 
of which the study of the Askey--Wilson
polynomials has been central.
Many of the properties of these
polynomials can be derived from
their terminating basic hypergeometric
representations, so an exhaustive
catalog of these representations, as contained here,
will be quite convenient for lookup.}

The Askey--Wilson polynomials can be defined in terms of 
terminating basic hypergeometric series 
(see, e.g., \eqref{2.12}), which in turn are defined in 
terms of a sum of products of $q$-Pochhammer symbols.
Using the properties of $q$-Pochhammer symbols, it is straightforward to 
replace $q\mapsto 1/q$ in the complex plane in order to obtain an extension 
of these polynomials with $|q|>1$. One often refers to these polynomials obtained
as $q^{-1}$ or $1/q$ polynomials. Since these algebraic factors 
are difficult to search on in the literature, we refer to this extension 
specifically as the $q$-inverse polynomials. 


It should however be noted that while the Askey--Wilson polynomials 
represent an infinite-family of orthogonal polynomials ($n\in\mathbb N_0$), 
orthogonal with respect to a weight function on $[-1,1]$ 
\cite[(14.1.2)]{Koekoeketal} (which gives restrictions on the values of 
the free parameters), the $q$-inverse Askey--Wilson polynomials represent 
a finite-family of basic hypergeometric orthogonal polynomials 
($n\in\{0,\ldots,N\}$, $N\in\mathbb N_0$)
(see e.g. \cite{IsmailZhang2019}). 
Not only that, but it is also known that the $q$-inverse Askey--Wilson 
polynomials are simply a scaled version of the Askey--Wilson polynomials 
with their parameters replaced by their reciprocals (see Remark 
\ref{invAWrem}, below).

In the sequel, from the basic terminating hypergeometric representations 
of the Askey--Wilson polynomials, we derive terminating basic hypergeometric 
representations
for the $q$-inverse Askey--Wilson polynomials 
\boro{(the Askey--Wilson polynomials with $q$ replaced by $1/q$).
From the terminating basic hypergeometric representations of the 
Askey--Wilson polynomials, one can easily 
derive transformation formulae for terminating basic hypergeometric functions.
}

The main focus of this survey paper will be to exhaustively describe the 
transformation identities for the terminating basic hypergeometric functions which
appear as representations for these polynomials. Some of these transformation 
identities are well-known in the literature, but we also give the transformation 
identities for these basic hypergeometric functions which are 
obtained by the symmetry of the polynomials under parameter interchange, 
and under the map $\theta\mapsto-\theta$, for $x=\cos\theta$.
\boro{It should be noted that the symmetry group $G$ of these functions
coincide with the symmetry group of the very-well poised ${}_8W_7$
which is a subgroup $WD_5$ of $WB_5$, the Weyl group of a root system of
type $B_n$. The order of $G$ is $5!2^4=1920$. For more details, see
\cite{VanderJeugtRao1999,KrattenthalerRao2004,Kajihara2014}.
}

\section{Preliminaries}
\label{Preliminaries}
We adopt the following set 
notations: $\mathbb N_0:=\{0\}\cup\mathbb N=\{0, 1, 2, \ldots\}$, and we use 
the sets $\mathbb Z$, $\mathbb R$,  $\mathbb C$ which represent 
the integers, real numbers  and complex numbers respectively, 
$\CCast:=\CC\setminus\{0\}$.
We also adopt the following notation and conventions.
Let ${\bf a}:=\{a_1, a_2, a_3, a_4\}$, $b, a_k\in\CC$, $k=1, 2, 3, 4$. 
Define
${\bf a}+b:=\{a_1+b,a_2+b,a_3+b,a_3+b\}$,
$a_{12}:=a_1a_2$,
$a_{13}:=a_1a_3$,
$a_{23}:=a_2a_3$,
$a_{123}:=a_1a_2a_3$,
$a_{1234}:=a_1a_2a_3a_4$, etc. 
Throughout the paper, we assume that the empty sum 
vanishes and the 
empty product 
is unity.

\begin{defn} \label{def:2.1}
Throughout this paper we adopt the following conventions for succinctly 
writing elements of lists. To indicate sequential positive and negative 
elements, we write
\[
\pm a:=\{a,-a\}.
\]
\noindent We also adopt an analogous notation
\[
\expe^{\pm i\theta}:=\{\expe^{i\theta},\expe^{-i\theta}\}.
\]
\noindent In the same vein, consider a finite 
sequence $f_s\in\CC$ with $s\in{\mathcal S}\subset \mathbb N$.
Then, the notation
$\{f_s\}$
represents the sequence of all complex numbers $f_s$ such that 
$s\in\SSS$.
Furthermore, consider some $p\in\SSS$, then the notation
$\{f_s\}_{s\ne p}$ represents the sequence of all complex numbers
$f_s$ such that $s\in\SSS\!\setminus\!\{p\}$.
In addition, for the empty list, $n=0$, we take
\[
\{a_1,\ldots,a_n\}:=\emptyset.
\]
\end{defn}

Consider $q\in\CC^\ast$ such that $|q|\ne 1$.
Define the sets 
$\Omega_q^n:=\{q^{-k}:n,k\in\mathbb N_0, 0\le k\le n-1\}$,
$\Omega_q:=\Omega_q^\infty=\{q^{-k}:k\in\mathbb N_0\}$.
In order to obtain our derived identities, we rely on properties 
of the $q$-Pochhammer symbol ($q$-shifted factorial). 
For any $n\in \mathbb N_0$, $a,q \in \CC$, 
the 
$q$-Pochhammer symbol is defined as
\begin{eqnarray}
&&\hspace{-3.3cm}\label{poch.id:1} 
(a;q)_n:=(1-a)(1-aq)\cdots(1-aq^{n-1}),\quad n\in\mathbb N_0.
\end{eqnarray}
One may also define
\begin{equation}
(a;q)_\infty:=\prod_{n=0}^\infty (1-aq^{n}), \label{poch.id:2}
\end{equation}
where $|q|<1$.
Furthermore, define
\[
(a;q)_b:=\frac{(a;q)_\infty}{(a q^b;q)_\infty}.
\]
where $a q^b\not \in \Omega_q$.
We will also use the common notational product conventions
\begin{eqnarray}
&&\hspace{-8cm}(a_1,...,a_k)_b:=(a_1)_b\cdots(a_k)_b,\nonumber\\
&&\hspace{-8cm}(a_1,...,a_k;q)_b:=(a_1;q)_b\cdots(a_k;q)_b.\nonumber
\end{eqnarray}

The following properties for the $q$-Pochhammer 
symbol can be found in Koekoek et al. 
\cite[(1.8.7), (1.8.10-11), (1.8.14), (1.8.19), 
(1.8.21-22)]{Koekoeketal}, namely for appropriate 
values of $q,a\in\CCast$ and $n,k\in\mathbb N_0$:
\begin{eqnarray}
\label{poch.id:3} &&\hspace{-6.5cm}
(a;q^{-1})_n=(a^{-1};q)_n(-a)^nq^{-\binom{n}{2}},
\\[2mm] 
\label{poch.id:4}
&&\hspace{-6.5cm}(a;q)_{n+k}=(a;q)_k(aq^k;q)_n 
= (a;q)_n(aq^n;q)_k,\\[2mm] 
\label{poch.id:5}&&\hspace{-6.5cm} (a;q)_n=(q^{1-n}/a;q)_n(-a)^nq^{\binom{n}{2}},\\[2mm]
\label{poch.id:6}&&\hspace{-6.5cm}(aq^{-n};q)_{k}=q^{-nk}
\frac{(q/a;q)_n}{(q^{1-k}/a;q)_n}(a;q)_k,\\[2mm]
\label{poch.id:8}&&\hspace{-6.5cm}(a^2;q^2)_n=(\pm a;q)_n,\\[2mm]
\label{poch.id:7}&&\hspace{-6.5cm}(a;q)_{2n}=(a,aq;q^2)_n=(\pm\sqrt{a},\pm\sqrt{qa};q)_n.
\end{eqnarray}
\noindent Observe that by using (\ref{poch.id:1}) and 
(\ref{poch.id:8}), one obtains
\begin{eqnarray}
&&\hspace{-6.5cm}(aq^n;q)_n=\frac{(\pm 
\sqrt{a},\pm \sqrt{aq};q)_n}{(a;q)_n},\quad 
a\not\in\Omega_q^n. 
\label{poch.id:9}
\end{eqnarray}

The basic hypergeometric series, which we 
will often use, is defined for
$q,z\in\CCast$ such that $|q|,|z|<1$, $s,r\in\mathbb N_0$, 
$b_j\not\in\Omega_q$, $j=1,...,s$, as
\cite[(1.10.1)]{Koekoeketal}
\begin{equation}
\qhyp{r}{s}{a_1,...,a_r}
{b_1,...,b_s}
{q,z}
:=\sum_{k=0}^\infty
\frac{(a_1,...,a_r;q)_k}
{(q,b_1,...,b_s;q)_k}
\left((-1)^kq^{\binom k2}\right)^{1+s-r}
z^k.
\label{2.11}
\end{equation}
\noindent Note that we refer to a basic hypergeometric
series as { $\ell$-balanced} if
$q^\ell a_1\cdots a_r=b_1\cdots b_s$, 
and { balanced} if $\ell=1$.
A basic hypergeometric series ${}_{r+1}\phi_r$ is { well-poised} if 
the parameters satisfy the relations
\[
qa_1=b_1a_2=b_2a_3=\cdots=b_ra_{r+1}.
\]
\noindent It is {very-well poised} if in addition, 
$\{a_2,a_3\}=\pm q\sqrt{a_1}$.

Similarly for terminating basic hypergeometric series 
which appear in basic hypergeometric orthogonal
polynomials, one has
\begin{equation}
\qhyp{r}{s}{q^{-n},a_1,...,a_{r-1}}
{b_1,...,b_s}{q,z}:=\sum_{k=0}^n
\frac{(q^{-n},a_1,...,a_{r-1};q)_k}{(q,b_1,...,b_s;q)_k}
\left((-1)^kq^{\binom k2}\right)^{1+s-r}z^k,
\label{2.12}
\end{equation}
where $b_j\not\in\Omega_q^n$, $j=1,...,s$.
Define the very-well poised 
basic hypergeometric series
${}_{r+1}W_r$ \cite[(2.1.11)]{GaspRah}
\begin{equation}
\label{rpWr}
{}_{r+1}W_r(b;a_4,\ldots,a_{r+1};q,z)
:=\qhyp{r+1}{r}{b,\pm q\sqrt{b},a_4,\ldots,a_{r+1}}
{\pm \sqrt{b},\frac{qb}{a_4},\ldots,\frac{qb}{a_{r+1}}}{q,z},
\end{equation} 
where $\sqrt{b},\frac{qb}{a_4},\ldots,\frac{qb}{a_{r+1}}\not\in\Omega_q$. 
When the very-well poised basic hypergeometric
series is terminating, then one has
\begin{equation}
{}_{r+1}W_r\left(b;q^{-n},a_5,\ldots,a_{r+1};q,z\right)
=\qhyp{r+1}{r}{b,\pm q\sqrt{b},q^{-n},a_5,\ldots,a_{r+1}}
{\pm \sqrt{b},q^{n+1}b,\frac{qb}{a_5},\ldots,\frac{qb}{a_{r+1}}}
{q,z},
\label{eq:2.13}
\end{equation}
where $\sqrt{b},\frac{qb}{a_5},\ldots,\frac{qb}{a_{r+1}}\not\in\Omega_q^n\cup\{0\}$.
The Askey--Wilson polynomials are intimately connected with 
the terminating very-well poised ${}_8W_7$, which is
given by
\begin{equation}
\label{VWP87}
{}_{8}W_7(b;q^{-n}\!,c,d,e,f;q,z)
=\qhyp{8}{7}{b,\pm q\sqrt{b},q^{-n},c,d,e,f}
{\pm \sqrt{b},q^{n+1}b,\frac{qb}{c},\frac{qb}{d},\frac{qb}{e},\frac{qb}{f}}{q,z},
\end{equation}
where $\sqrt{b},\frac{qb}{c},\frac{qb}{d},\frac{qb}{e},\frac{qb}{f}\not\in\Omega_q^n\cup\{0\}$.

\medskip


Some classical transformations for basic hypergeometric series which 
we will use include Watson's $q$-analog of Whipple's theorem which 
relates a terminating balanced ${}_4\phi_3$ to a terminating very-well 
poised ${}_8W_7$ (cf. \cite[(17.9.15)]{NIST:DLMF})
\begin{equation}
\label{WatqWhipp}
\qhyp43{q^{-n},a,b,c}{d,e,f}{q,q}=
\frac{\left(\frac{de}{ab},\frac{de}{ac};q\right)}
{\left(\frac{de}{a},\frac{de}{abc};q\right)}
{}_8W_7\left(\frac{de}{qa};q^{-n},\frac{d}{a},\frac{e}{a},b,c;q,\frac{qa}{f}\right),
\end{equation}
where \yoro{$q^{1-n}abc=def$}.

In \cite[Exercise 1.4ii]{GaspRah}, one finds the inversion formula for
terminating basic hypergeometric series.

\begin{thm}[Gasper and Rahman (1990)] \label{thm:2.2}
Let $m, n, k, r, s\in\mathbb N_0$, $0\le k\le r$, $0\le m\le s$, $a_k\in\CC$, $b_m\not \in\Omega^n_q$,
$q\in\CC^\ast$ such that $|q|\ne 1$.
Then,
\begin{eqnarray}
\qhyp{r+1}{s}{q^{-n},a_1,...,a_r}{b_1,...,b_s}{q,z}=&& \frac{(a_1,...,a_r;q)_n}
{(b_1,...,b_s;q)_n}\left(\frac{z}{q}\right)^n\left((-1)^nq^{\binom{n}{2}}\right)^{s-r-1} \nonumber \\
&& \label{inversion} \times \sum_{k=0}^n \frac{\left(q^{-n},\frac{q^{1-n}}{b_1},...,\frac{q^{1-n}}{b_s};q\right)_k}
{\left(q,\frac{q^{1-n}}{a_1},...,\frac{q^{1-n}}{a_r};q\right)_k} \left(\frac{b_1 \cdots b_s}{a_1 \cdots a_r}
\frac{q^{n+1}}{z}\right)^k.
\end{eqnarray}
\end{thm}
\begin{cor}\label{cor:2.4}
Let $n,r\in\mathbb N_0$, $q\in\CC^\ast$ such that $|q|\ne 1$, and 
for $0\le k\le r$, let $a_k, b_k\not\in\Omega^n_q\cup\{0\}$.
Then,
\begin{equation}\label{cor:2.4:r1}
\qhyp{r+1}{r}{q^{-n},a_1,\ldots,a_r}{b_1,\ldots,b_r}{q,z}\!=\!
(-1)^nq^{-\binom{n}{2}}
\frac{(a_1,\ldots,a_r;q)_n}
{(b_1,\ldots,b_r;q)_n}
\qhyp{r+1}{r}{q^{-n},
\frac{q^{1-n}}{b_1},\ldots,
\frac{q^{1-n}}{b_r}}
{\frac{q^{1-n}}{a_1},\ldots,
\frac{q^{1-n}}{a_r}}{q,
\frac{q^{n+1}}{z}\frac{b_1\cdots b_r}
{a_1\cdots a_r}}.
\end{equation}
\end{cor}
\begin{proof}
Take $r=s$, in \eqref{inversion} and using the definition \eqref{2.11}
completes the proof.
\end{proof}
\zoro{
\noindent Note that in Corollary \ref{cor:2.4}
if the terminating basic hypergeometric
series on the left-hand side is balanced
then the argument of the terminating
basic hypergeometric series on the right-hand side is $q^2/z$.
Applying Corollary \ref{cor:2.4} to the definition
of ${}_{r+1}W_r$, we obtain the following result for a terminating very-well 
poised basic hypergeometric series ${}_{r+1}W_r$.
}
\begin{cor} \label{cor:2.5}
Let 
$n\in\mathbb N_0$, $b,a_k,q,z\in\CCast$, $\sqrt{b},q^{n+1}b,\frac{q b}{a_k},
\frac{q^{1-n}}{b},\frac{q^{1-n}}{a_k}\not\in\Omega_q^n$, $k=5,\ldots,r+1$. 
Then one has the following transformation formula for a very-well poised terminating 
basic hypergeometric series:
\begin{eqnarray}
&&\hspace{-0.4cm}{}_{r+1}W_r\left(b;q^{-n},a_5,\ldots,a_{r+1};q,z\right)=
q^{-\binom{n}{2}}\left(\frac{-z}{q}\right)^n
\frac{(\pm q\sqrt{b},b,a_5,\ldots,a_{r+1};q)_n}
{\left(\pm \sqrt{b},q^{n+1}b,\frac{qb}{a_5},\ldots,\frac{qb}{a_{r+1}};q\right)_n}\nonumber\\
&&\hspace{5cm}\times{}_{r+1}W_r\left(\frac{q^{-2n}}{b};q^{-n},
\frac{q^{-n}a_5}{b},\ldots,\frac{q^{-n}a_{r+1}}{b};q,\frac{q^{2n+r-3}b^{r-3}}{(a_5\cdots a_{r+1})^2 z}\right).\nonumber
\end{eqnarray}
\end{cor}
\begin{proof}
Use Corollary \ref{cor:2.4} and \eqref{eq:2.13}.
\end{proof}

An interesting and useful consequence of this formula is the $r=7$ special case,
\begin{eqnarray}
&&{}_8W_7\left(b;q^{-n},c,d,e,f;q,z\right)
=q^{-\binom{n}{2}} \left(\frac{-z}{q}\right)^n
\frac{\left(\pm q\sqrt{b},b,c,d,e,f;q\right)_n}
{\left(\pm\sqrt{b},q^{n+1}b,\frac{qb}{c},\frac{qb}{d},\frac{qb}{e},\frac{qb}{f};q\right)_n}\nonumber\\
&&\hspace{4.5cm}\times
{}_8W_7\left(\frac{q^{-2n}}{b};q^{-n},
\frac{q^{-n}c}{b},
\frac{q^{-n}d}{b},
\frac{q^{-n}e}{b},
\frac{q^{-n}f}{b};
q,\frac{q^{2n+4}b^4}{z(cdef)^2}
\right).
\label{inv87}
\end{eqnarray}
\noindent Note that in the case when the one
obtains an ${}_8W_7$ from a balanced 
${}_4\phi_3$ using \eqref{WatqWhipp},
then $q^{2n+4}b^4/(z(cdef)^2)=z$.

Another equality we can use is the following connecting relation between 
basic hypergeometric series on $q$, and on $q^{-1}$:
\begin{eqnarray} 
&&\qhyp{r+1}{r}{q^{-n},a_1,...,a_r}{b_1,...,b_r}{q,z}=
\qhyp{r+1}{r}{q^{n}, a^{-1}_1,..., a^{-1}_r}
{b^{-1}_1, ..., b^{-1}_r}{q^{-1}, 
\dfrac{a_1 a_2\cdots a_r}
{b_1 b_2\cdots b_r}\dfrac{z}{q^{n+1}}}\nonumber\\
&&\hspace{2cm}=\frac{(a_1,\ldots,a_r;q)_n}
{(b_1,\ldots,b_r;q)_n}
\left(-\frac zq\right)^nq^{-\binom{n}{2}}
\qhyp{r+1}{r}{q^{-n},
\frac{q^{1-n}}{b_1},...,
\frac{q^{1-n}}{b_r}}
{\frac{q^{1-n}}{a_1},...,
\frac{q^{1-n}}{a_r}}
{q,\frac{b_1\cdots b_r}{a_1\cdots a_r}\frac{q^{n+1}}{z}}.
\label{qtopiden} 
\end{eqnarray}
\noindent In order to understand the procedure for obtaining the $q$-inverse analogues of the 
basic hypergeometric orthogonal polynomials studied in this
manuscript, let's consider a special case in detail. 
Let $f_r(q):=f_r(q;z(q);{\bf a}(q),{\bf b}(q))$ be 
defined as 
\begin{equation}
\label{freqs}
f_r(q):=g_r(q)
\qhyp{r+1}{r}{q^{-n},{\bf a}(q)}{{\bf b}(q)}{q,z(q)},
\end{equation}
where 
\[
\left.
\begin{array}{c}
{\bf a}(q):=\left\{{a_1(q)},\ldots,{a_r(q)}\right\} \\[0.3cm]
{\bf b}(q):=\left\{{b_1(q)},\ldots,{b_r(q)}\right\}
\end{array}
\right\},
\]
which will suffice for instance, for the study of the terminating basic hypergeometric 
representations for the Askey--Wilson polynomials. 
In order to obtain the corresponding $q$-inverse hypergeometric representations 
of $f_r(q)$, one only needs to consider the corresponding $q$-inverted function:
\begin{equation}
f_r(q^{-1})=g_r(q^{-1})
\qhyp{r+1}{r}{q^n,{\bf a}(q^{-1})}{{\bf b}(q^{-1})}{q^{-1},z(q^{-1})}.
\label{invertedrep}
\end{equation}
\begin{thm} \label{thm:2.6}
Let $r,k\in\mathbb N_0$, $0\le k\le r$, $a_k(q)\in\CC$, $b_k(q)\not\in\Omega_q$, 
$q\in\CC^\ast$ such that $|q|\ne 1$, $z(q)\in\CC$.
Define 
a multiplier function $g_r(q):=g_r(q;z(q);{\bf a}(q);{\bf b}(q))$ which is not 
of basic hypergeometric type (some multiplicative combination of powers and $q$-Pochhammer symbols), 
and $z(q):=z(q;{\bf a}(q);{\bf b}(q))$. Then, define $f_r(q)$ as in \eqref{freqs}, and one has
\begin{equation}
\label{termreprr}
f_r(q^{-1})
=g_r(q^{-1})\qhyp{r+1}{r}{q^{-n},{\bf a}^{-1}(q^{-1})}{{\bf b}^{-1}(q^{-1})}{q,\frac{q^{n+1}a_1(q^{-1})\cdots a_r(q^{-1}) z(q^{-1})}{b_1(q^{-1})\cdots b_r(q^{-1})}},
\end{equation}
where 
\[
\left.
\begin{array}{c}
{\bf a}^{-1}(q^{-1})=\left\{\frac{1}{a_1(q^{-1})},\ldots,\frac{1}{a_r(q^{-1})}\right\} \\[0.3cm]
{\bf b}^{-1}(q^{-1})=\left\{\frac{1}{b_1(q^{-1})},\ldots,\frac{1}{b_r(q^{-1})}\right\}
\end{array}
\right\}.
\]
\end{thm}

\begin{proof}
Using the identity \eqref{poch.id:3} repeatedly with the definition 
\eqref{2.11}, in \eqref{invertedrep}, obtains the 
$q$-inverted terminating representation \eqref{termreprr} which corresponds 
to the original terminating basic hypergeometric representation
\eqref{freqs}. This completes the proof.
\end{proof}

\medskip

\noindent We will obtain new transformations for basic hypergeometric orthogonal
polynomials by taking advantage of the following remark.

\begin{rem} \label{rem:2.7}
Since $x=\cos\theta$ is an even function of $\theta$, 
all polynomials in $\cos\theta$ will be invariant under the 
map $\theta\mapsto-\theta$.
\end{rem}

\begin{rem} \label{rem:2.8}
Observe in the following discussion we will often be referring
to a collection of constants $a,b,c,d,e,f$. In such cases, which will be
clear from context, then the constant $e$ should not be confused
with Euler's number, the base of the natural logarithm, i.e.,
$\log\expe=1$.
\end{rem}

\section{The Askey--Wilson Polynomials}
\label{askeywilson}

Define the sets ${\bf 4}:=\{1,2,3,4\}$, 
${\bf a}:=(a_1,a_2,a_3,a_4)$, $a_k\in\CCast$, $k\in{\bf 4}$, and
$x=\cos\theta\in[-1,1]$.
The Askey--Wilson polynomials $p_n(x;{\bf a}|q)$ are a family of polynomials symmetric 
in four free parameters. 
These polynomials have a
long and in-depth history and their properties have been studied in detail. 
The basic hypergeometric series representation of the Askey--Wilson polynomials fall into four main categories: (1) terminating ${}_4\phi_3$ representations; (2) terminating ${}_8W_7$ representations; (3) nonterminating ${}_8W_7$ representations; and (4) nonterminating ${}_4\phi_3$ representations.

One may obtain alternative 
nonterminating representations of the Askey--Wilson polynomials using 
\cite[(2.10.7)]{GaspRah}, namely
\begin{eqnarray}
&&\hspace{-1.5cm}\qhyp43{q^{-n},q^{n-1}a_{1234},a_p\expe^{\pm i\theta}}
{\{a_{ps}\}_{s\ne p}}{q,q}=
\frac{\left(q^{1-n}\expe^{2i\theta},\frac{q^{1-n}}{a_{tu}},
\frac{q^{2-n}\expe^{i\theta}}{a_{prt}},\frac{q^{2-n}\expe^{i\theta}}
{a_{pru}};q\right)_\infty}{\left(\frac{q^{1-n}\expe^{i\theta}}
{a_t},\frac{q^{1-n}\expe^{i\theta}}{a_u},
\frac{q^{2-n}\expe^{2i\theta}}{a_{pr}},\frac{q^{2-n}}
{a_{1234}};q\right)_\infty}\nonumber\\
\label{4ph3to8W7}&&\hspace{3.5cm}\times
{}_8W_7\left(\frac{q^{1-n}\expe^{2i\theta}}{a_{pr}}
;\frac{q^{1-n}}{a_{pr}},\frac{q\expe^{i\theta}}
{a_p},\frac{q\expe^{i\theta}}{a_r},a_t\expe^{i\theta},a_u\expe^{i\theta}
;q,\frac{q^{1-n}}{a_{tu}}\right),
\end{eqnarray}
provided
$|q^{1-n}/a_{tu}|<1$.
However, there exist no
fixed values of $|a_k|, |q|<1$ such that
this very-well-poised ${}_8W_7$ is convergent for all
$n\in\mathbb N_0$. 
On the other hand, it is possible to find nonterminating ${}_8W_7$ 
representations which are convergent for all $n\in\mathbb N_0$. 
It is also possible to find nonterminating ${}_4\phi_3$ representations 
of the Askey--Wilson polynomials using Bailey's transformation
of a very-well-poised ${}_8W_7$ 
\cite[(17.9.16)]{NIST:DLMF}
\vspace{12pt}
\begin{eqnarray}
&&\hspace{-0.9cm}{}_8W_7\left(b;a,c,d,e,f;q,\frac{q^2b^2}{acdef}\right)
=\frac{(qb,\frac{qb}{de},\frac{qb}{df},\frac{qb}{ef};q)_\infty}
{(\frac{qb}{d},\frac{qb}{e},\frac{qb}{f},\frac{qb}{def};q)_\infty}
\qhyp43{\frac{qb}{ac},d,e,f}{\frac{qb}{a},\frac{qb}{c},\frac{def}{b}}
{q,q}\nonumber\\&&\label{8W7to4ph3}
\hspace{3cm}+\frac{(qb,\frac{qb}{ac},d,e,f,
\frac{q^2b^2}{adef},\frac{q^2b^2}{cdef};q)_\infty}
{(\frac{qb}{a},\frac{qb}{c},\frac{qb}{d},\frac{qb}{e},\frac{qb}{f},
\frac{q^2b^2}{acdef},\frac{def}{qb};q)_\infty}
\qhyp43{\frac{qb}{de},\frac{qb}{df},\frac{qb}{ef},\frac{q^2b^2}{acdef}}
{\frac{q^2b^2}{adef},\frac{q^2b^2}{cdef},\frac{q^2b}{def}}{q,q}.
\end{eqnarray}
\noindent However, these nonterminating representations will not be further 
discussed in this paper.

Different series representations are useful for obtaining different 
properties and formulae for these polynomials. So it is very useful 
to have at hand an exhaustive list.
The discussion contained in this section is an attempt to summarize, 
in an in-depth manner, an exhaustive description of the representation 
and transformation properties of the {terminating}
${}_4\phi_3$ and ${}_8W_7$ basic hypergeometric representations of 
the Askey--Wilson polynomials. 

\subsection{The Askey--Wilson Polynomial Representations}

\begin{thm} \label{thm:3.1}
Let $n\in\mathbb N_0$, $p,s,r,t,u\in {\bf 4}$, $p,r,t,u$ distinct and fixed, 
$q\in\CCast$ such that $|q|\ne 1$.
Then, the Askey--Wilson polynomials have the following terminating basic 
hypergeometric series representations given by:
\begin{eqnarray}
\hspace{-0.50cm}
\label{aw:def1} p_n(x;{\bf a}|q) &&:= a_p^{-n} \left(\{a_{ps}\}_{s \neq p};q\right)_n 
\qhyp43{q^{-n},q^{n-1}a_{1234}, a_p\expe^{\pm i\theta}}{\{a_{ps}\}_{s \neq p}}{q,q} \\
\label{aw:def2} &&=q^{-\binom{n}{2}} (-a_p)^{-n} 
\frac{\left(\frac{a_{1234}}{q};q\right)_{2n}
\left(a_p \expe^{\pm i\theta};q\right)_n}
{\left(\frac{a_{1234}}{q};q\right)_n}
\qhyp43{q^{-n}, \left\{\frac{q^{1-n}}{a_{ps}}\right\}_{s \neq p}}
{\frac{q^{2-2n}}{a_{1234}}, \frac{q^{1-n}\expe^{\pm i\theta}}{a_p}}{q,q} \\
\label{aw:def3} &&=\expe^{in\theta} \left(a_{pr},a_t\expe^{-i\theta},a_u\expe^{-i\theta};q\right)_n 
\qhyp43{q^{-n}, a_p\expe^{i\theta}, a_r\expe^{i\theta}, \frac{q^{1-n}}{a_{tu}}}{a_{pr}, \frac{q^{1-n}\expe^{i\theta}}{a_t}, 
\frac{q^{1-n}\expe^{i\theta}}{a_u}}{q,q} \\
&&=\expe^{in\theta}
\dfrac{\left(\frac{a_{1234}}{q};q\right)_{2n}\left(\left\{a_s\expe^{-i\theta}\right\}_{s\ne p}
,\frac{a_{1234}\,\expe^{-i\theta}}{q a_p};q\right)_{n}}
{\left(\frac{a_{1234}}{q};q\right)_{n}
\left(\frac{a_{1234}\,\expe^{-i\theta}}{q a_p};q\right)_{2n}}
\nonumber \\
&&\label{aw:def6}
\hspace{3cm}\times\,{}_8W_7\left(
\frac{q^{1-2n}a_p\expe^{i\theta}}{a_{1234}};q^{-n},
\left\{\frac{q^{1-n}a_{ps}}{a_{1234}}\right\}_{s\ne p},
a_p\expe^{i\theta};q,
\frac{q\expe^{i\theta}}{a_p}\right)\\
&&\label{aw:def7}
=\expe^{in\theta}
\dfrac{\left(a_p\expe^{-i\theta},\{\frac{a_{1234}}{a_{ps}}\}_{s\ne p};q\right)_n}
{\left(\frac{a_{1234}\,\expe^{i\theta}}{a_p};q\right)_{n}}
{}_8W_7\left(
\frac{a_{1234}\expe^{i\theta}}{qa_p};q^{-n},
\{a_s\expe^{i\theta}\}_{s\ne p},q^{n-1}a_{1234};q,
\frac{q\expe^{-i\theta}}{a_p}\right)\\
&&\label{aw:def5}=a_p^{-n}\dfrac{\left(a_{pt},a_{pu},a_r\expe^{\pm i\theta};q\right)_n}{\left(\frac{a_r}{a_p};q\right)_n} 
\,{}_8W_7\left(
\frac{q^{-n}a_p}{a_r};q^{-n},\!
\frac{q^{1-n}}{a_{rt}}, \frac{q^{1-n}}{a_{ru}}, a_p\expe^{\pm i\theta};q,q^n a_{tu}\right)\\
\label{aw:def4} &&=\expe^{in\theta} \frac{\left(\{a_s\expe^{-i\theta}\};q\right)_n}{\left(\expe^{-2i\theta};q\right)_n} 
\,{}_8W_7\left(
q^{-n}\expe^{2i\theta};q^{-n},
\{a_s\expe^{i\theta}\};
q,\frac{q^{2-n}}{a_{1234}}
\right).
\end{eqnarray}
\label{AWthm}
\end{thm}

\begin{proof}
The standard definition of the Askey--Wilson polynomials \eqref{aw:def1} is found in 
many places including (\cite[(14.1.1)]{Koekoeketal}).
The representation \eqref{aw:def2} can be derived by applying \eqref{inversion} 
to \eqref{aw:def1}. It also follows
by using \cite[second equality in (17.9.14)]{NIST:DLMF} with $a=a_{1234}q^{n-1}$, 
$\{b,c\}=a_p\expe^{\pm i\theta}$, $\{d,e,f\}=\{a_{ps}\}_{s\ne p}$.
One can obtain \eqref{aw:def3} by starting with \eqref{aw:def1} and $p\leftrightarrow u$
using \cite[second equality in (17.9.14)]{NIST:DLMF}  
with $\{a,b\}=a_u\expe^{\mp i\theta}$, $c=a_{1234}q^{n-1}$, $\{d,e,f\}=\{a_{us}\}_{s\ne u}$,
and \cite[(1.8.14)]{Koekoeketal}.
The representation \eqref{aw:def6} follows by
from \eqref{aw:def7} by using the inversion
formula \eqref{inv87}.
The representation \eqref{aw:def7} follows from
\eqref{aw:def1} with $p\leftrightarrow r$, using Watson's $q$-analogue of Whipple's
theorem \eqref{WatqWhipp} and mapping $\theta\mapsto-\theta$.
The representation \eqref{aw:def5} follows by using \cite[(III.19)]{GaspRah} 
directly with \eqref{aw:def1}.
The representation \eqref{aw:def4} follows from \eqref{aw:def1} 
and \cite[ (III.15), (III.19)]{GaspRah} (see also \cite[\S 14.1]{KoornwinderKLSadd}).
This completes the proof.
\end{proof}

Note that when Corollary \ref{cor:2.4} is applied to \eqref{aw:def1}
one obtains \eqref{aw:def2}, and when one 
applies it to \eqref{aw:def3}, one obtains
the same formula back with $\theta\mapsto-\theta$ and $\{r,s\}\leftrightarrow\{t,u\}$.

\begin{rem} \label{rem:3.2}
Applying \eqref{inversion} to \eqref{aw:def3},
\eqref{aw:def6}, \eqref{aw:def7}, \eqref{aw:def4} simply takes 
$\theta \mapsto -\theta$, and applying it to \eqref{aw:def5} interchanges $a_p$ and $a_r$. 
Mapping $\theta \mapsto -\theta$ may give additional representations, however those are omitted.
\end{rem}


\subsection{Terminating 4-Parameter Symmetric Transformations}

\zoro{
\begin{cor}
\label{cor:3.3}
Let $n\in\mathbb N_0$, $b,c,d,e,f,q\in\CCast$, 
such that $|q|\ne 1$.
Then, one has 
the following transformation
formulas for a terminating ${}_8W_7$:
\begin{eqnarray}
&&\label{aw:def4to5}\hspace{-0.45cm}{}_8W_7\left(b;q^{-n},c,d,e,f;q,\frac{q^{n+2}b^2}{cdef}\right)\\
&&\hspace{-0.2cm}\label{cor3.5a.1}=q^{\binom{n}{2}}\!
\left(\frac{-q^2b^2}{cdef}\right)^{\!n}\!\!\!\dfrac{\left(qb,b, c, d, e, f;q\right)_n}
{(b;q)_{2n}\!
\left(\frac{qb}{c},\!\frac{qb}{d},\!\frac{qb}{e},\!\frac{qb}{f};q\right)_{\!n}} 
{}_8W_7\!\left(\frac{q^{-2n}}{b};
q^{-n},\!\frac{q^{-n}c}{b},\!
\frac{q^{-n}d}{b},\!\frac{q^{-n}e}{b},\!\frac{q^{-n}f}{b};
q,\frac{q^{n+2}b^2}{cdef}
\right)\\
&&\hspace{-0.2cm}\label{cor3.5a.2}=\frac{\left(\frac{qb}{ce},\frac{qb}{cf},qb,d;q\right)_n}
{\left(\frac{qb}{c},\frac{qb}{e},\frac{qb}{f},\frac{d}{c};q\right)_n}
{}_8W_7\left(\frac{q^{-n}c}{d};q^{-n},\frac{q^{-n}c}{b},
\frac{qb}{de},\frac{qb}{df},c;q,\frac{ef}{b}\right)\\
&&\hspace{-0.2cm}\label{cor3.5a.7}
=\frac{\left(\frac{qb}{de},\frac{qb}{df},\frac{qb}{ef},qb;q\right)_n}
{\left(\frac{qb}{def},\frac{qb}{d},\frac{qb}{e},\frac{qb}{f};q\right)_n}
{}_8W_7\left(\frac{q^{-n-1}def}{b};q^{-n},d,e,f,\frac{q^{-n-1}cdef}{b^2};q,\frac{q}{c}\right)
\\
&&\hspace{-0.2cm}\label{cor3.5a.7b}
=\frac{\left(
\frac{q^2b^2}{cdef},qb,d,e,f;q\right)_n}
{\left(\frac{def}{qb},\frac{qb}{c},\frac{qb}{d},\frac{qb}{e},\frac{qb}{f};q\right)_n}
{}_8W_7\left(\frac{q^{1-n}b}{def};q^{-n},\frac{q^{-n}c}{b},
\frac{qb}{de},\frac{qb}{df},\frac{qb}{ef};q,\frac{q}{c}\right)
\\[-0.0cm]
&&\hspace{-0.2cm}\label{cor3.5a.6}
=\frac{\left(\frac{q^2b^2}{cdef},{qb};q\right)_n}
{\left(\frac{qb}{c},\frac{q^2b^2}{def};q\right)_n}
{}_8W_7\left(\frac{qb^2}{def};q^{-n},\frac{qb}{de},
\frac{qb}{df},\frac{qb}{ef},c;q,\frac{q^{n+1}{b}}{c}\right)
\\
&&\hspace{-0.2cm}\label{cor3.5a.7c}
=q^{\binom{n}{2}}
\left(-\frac{qb}{c}\right)^n\frac{\left(
\frac{qb^2}{def},
\frac{qb}{ef},\frac{qb}{de},\frac{qb}{df},qb,c;q\right)_n}
{\left(\frac{qb^2}{def};q\right)_{2n}\left(\frac{qb}{c},\frac{qb}{d},\frac{qb}{e},\frac{qb}{f};q\right)_n}
\nonumber\\[-0.00cm]
&&\hspace{3cm}\times{}_8W_7\left(\frac{q^{-2n-1}def}{b^2};q^{-n},\frac{q^{-n}d}{b},
\frac{q^{-n}e}{b},\frac{q^{-n}f}{b},\frac{q^{-n-1}cdef}{b^2};q,\frac{q^{n+1}b}{c}\right)
\\
&&\hspace{-0.2cm}\label{cor3.5a.3}=\frac{\left(\frac{qb}{cd},qb;q\right)_n}{\left(\frac{qb}{c},\frac{qb}{d};q\right)_n}\qhyp43{q^{-n},\frac{qb}{ef},c,d}{\frac{q^{-n}cd}{b},\frac{qb}{e},\frac{qb}{f}}{q,q}\\
&&\hspace{-0.2cm}\label{cor3.5a.4}=
\left(\frac{qb}{ef}\right)^n
\dfrac{\left(\frac{qb}{cd},qb,e, f; q\right)_n}{\left(\frac{qb}{c},\frac{qb}{d},\frac{qb}{e}, \frac{qb}{f};q\right)_n} 
\qhyp43{q^{-n}, \frac{q^{-n}c}{b}, 
\frac{q^{-n}d}{b}, \frac{qb}{ef}}{\frac{q^{-n}cd}b, \frac{q^{1-n}}{e}, \frac{q^{1-n}}{f}}{q,q}\\
&&\hspace{-0.2cm}\label{cor3.5a.5}=\frac{\left(\frac{q^2b^2}{cdef},qb,c;q\right)_n}
{\left(\frac{qb}{d},\frac{qb}{e},\frac{qb}{f};q\right)_n}
\qhyp43{q^{-n},\frac{qb}{cd},\frac{qb}{ce},\frac{qb}{cf}}
{\frac{q^2b^2}{cdef},\frac{q^{1-n}}{c},\frac{qb}{c}}{q,q}\\
&&\hspace{-0.2cm}\label{cor3.5a.6b}=c^n\frac{\left(\frac{qb}{cd},\frac{qb}{ce},\frac{qb}{cf},qb;q\right)_n}
{\left(\frac{qb}{c},\frac{qb}{d},\frac{qb}{e},\frac{qb}{f};q\right)_n}
\qhyp43{q^{-n},\frac{q^{-n-1}cdef}{b^2},\frac{q^{-n}c}{b},c}{\frac{q^{-n}cd}{b},\frac{q^{-n}ce}{b},\frac{q^{-n}cf}{b}}{q,q}.
\end{eqnarray}
\end{cor}
}

\begin{proof}
Start with Theorem \ref{thm:3.1} and set $\expe^{2i\theta}=q^nb$, $a_p=q^{-\frac{n}{2}}\frac{c}{\sqrt{b}}$,
$a_r=q^{-\frac{n}{2}}\frac{d}{\sqrt{b}}$,
$a_t=q^{-\frac{n}{2}}\frac{e}{\sqrt{b}}$,
$a_u=q^{-\frac{n}{2}}\frac{f}{\sqrt{b}}$,
\yoro{setting $\theta\mapsto-\theta$ where
necessary.}
Then, multiply every formula by the factor
\[
{\sf A}_n(b,c,d,e,f|q):=
\frac{q^{2\binom{n}{2}}
(-1)^n(qb)^\frac{5n}{2}\left(qb;q\right)_n
}
{
(cdef)^n\left(\frac{qb}{c},\frac{qb}{d},
\frac{qb}{e},\frac{qb}{f};q\right)_n
}.
\]
\noindent With simplification, this completes the proof.
\end{proof}

\begin{rem} \label{rem:3.4}
Notice that in Corollary \ref{cor:3.3}, our order
of the representations begins with the principal ${}_8W_7$ representation 
in which the symmetry in the parameters
$c,d,e,f$ is evident and ends with the representation 
corresponding to the classical 
${}_4\phi_3$ basic hypergeometric representation of the Askey--Wilson 
polynomials \eqref{aw:def1}. On the other hand, in Theorem \ref{AWthm},
we have reversed the order of the corresponding 
representations. The reason why we have used the
ordering as such is because the ${}_4\phi_3$ representation
of the Askey--Wilson polynomials  \eqref{aw:def1} is historically first 
(see \cite[(1.8)]{IsmailWilson82}  and the memoir
\cite[(1.15)]{AskeyWilson85} and is certainly the most
common (see e.g., \cite[(14.1.1)]{Koekoeketal})).
The Askey--Wilson polynomials are
symmetric in their four parameters, the ${}_8W_7$ 
representation in which this symmetry is evident
demonstrates this symmetry.
On the other hand, the 
polynomial nature of the 
Askey--Wilson polynomials 
is not clearly evident
from the ${}_8W_7$ representation. In the first ${}_4\phi_3$ representation, 
the polynomial nature of evident.
\end{rem}

\subsection{Terminating 4-Parameter Symmetric Interchange Transformations}

The evidence that the first (and second) ${}_8W_7$ in Corollary \ref{cor:3.3} are symmetric
in the variables $c,d,e,f$ is clear. Therefore, all of the
formulas in this corollary are invariant under
the interchange of any two of those variables. This is true whether 
the symmetry between those variables is
evident in the corresponding mathematical expression or not. 
Perhaps, the most famous parameter interchange transformation 
of this sort is Sears' balanced 
${}_4\phi_3$ transformations \cite[(17.9.14)]{NIST:DLMF}  which 
demonstrate the invariance (and provide specific transformation
formulas) of the Askey--Wilson polynomials under parameter interchange. 
Other interesting parameter interchange transformations of this type 
can be obtained, such as by \eqref{cor3.5a.2} with
$c\leftrightarrow d$ (preserves the argument), 
$c\leftrightarrow e$, 
$c\leftrightarrow f$, 
$d\leftrightarrow e$, 
$d\leftrightarrow f$ 
interchanged (the invariance under the interchange
$e\leftrightarrow f$ is evident). Furthermore, when 
the symmetry within a set of variables is
evident in the transformation corollaries presented below,
then due to this symmetry, non-trivial 
transformation formulas can be obtained
by equating the two expressions with certain variables interchanged. 

In this
subsection we present the entirety of all of the parameter interchange
transformations for terminating basic hypergeometric transformations 
which arise from the Askey--Wilson polynomials.

\zoro{
\begin{cor}
\label{cor:3.5}
Let $n\in\mathbb N_0$, $q,b,c,d,e,f\in\CCast$,
such that $|q|\ne 1$.
Then, one has 
the following
parameter interchange transformations for a terminating ${}_8W_7$:
\begin{eqnarray}
&&\hspace{-2.2cm}\label{cor3.5:r1}
{}_8W_7\left(\frac{q^{-n}c}{d};q^{-n},\frac{q^{-n}c}{b},
\frac{qb}{de},\frac{qb}{df},c;q,\frac{ef}{b}\right)\\
&&\hspace{-0.2cm}\label{cor3.5:r2}=
\frac{\left(\frac{qb}{de},\frac{qb}{df},\frac{qb}{c},\frac{d}{c},c;q\right)_n}
{\left(\frac{qb}{ce},\frac{qb}{cf},\frac{qb}{d},\frac{c}{d},d;q\right)_n}\,
{}_8W_7\left(\frac{q^{-n}d}{c};q^{-n},\frac{q^{-n}d}{b},
\frac{qb}{ce},\frac{qb}{cf},\boro{d};q,\frac{ef}{b}\right)\\
&&\hspace{-0.2cm}\label{cor3.5:r3}=
\boro{\frac
{\left(\frac{qb}{cd},\frac{qb}{e},\frac{d}{c},e;q\right)_n}
{\left(\frac{qb}{ce},\frac{qb}{d},\frac{e}{c},d;q\right)_n}\,}
{}_8W_7\left(\frac{q^{-n}c}{e};q^{-n},\frac{q^{-n}c}{b},
\frac{qb}{ed},\frac{qb}{ef},c;q,\frac{df}{b}\right)\\
&&\hspace{-0.2cm}\label{cor3.5:r4}=
\boro{\frac{\left(\frac{qb}{ed},\frac{qb}{ef},\frac{qb}{c},\frac{{d}}{c},c;q\right)_n}
{\left(\frac{qb}{c{e}},\frac{qb}{cf},\frac{qb}{{d}},\frac{c}{e},{d};q\right)_n}\,}
{}_8W_7\left(\frac{q^{-n}e}{c};q^{-n},\frac{q^{-n}e}{b},
\frac{qb}{cd},\frac{qb}{cf},e;q,\frac{df}{b}\right)\\
&&\hspace{-0.2cm}\label{cor3.5:r5}=
\boro{\frac{\left(\frac{qb}{{cd}},\frac{qb}{{f}},\frac{{d}}{{c}},f;q\right)_n}
{\left(\frac{qb}{{c}f},\frac{qb}{{d}}\frac{f}{c},{d};q\right)_n}\,}
{}_8W_7\left(\frac{q^{-n}c}{f};q^{-n},\frac{q^{-n}c}{b},
\frac{qb}{ef},\frac{qb}{df},c;q,\frac{de}{b}\right)
\end{eqnarray}
\begin{eqnarray}
&&\boro{\hspace{-0.2cm}\label{cor3.5:r6}=
\frac{\left(\frac{qb}{fd},\frac{qb}{fe},\frac{qb}{c},\frac{d}{c},c;q\right)_n}
{\left(\frac{qb}{ce},\frac{qb}{cf},\frac{qb}{d},\frac{c}{f},d;q\right)_n}
{}_8W_7\left(\frac{q^{-n}f}{c};q^{-n},\frac{q^{-n}f}{b},
\frac{qb}{cd},\frac{qb}{ce},f;q,\frac{de}{b}\right)}\\
&&\hspace{-0.2cm}\label{cor3.5:r7}\boro{=\frac
{\left(\frac{qb}{ef},\frac{d}{c};q\right)_n}
{\left(\frac{qb}{cf},\frac{d}{e};q\right)_n}
{}_8W_7\left(\frac{q^{-n}e}{d};q^{-n},\frac{q^{-n}e}{b},
\frac{qb}{dc},\frac{qb}{df},e;q,\frac{cf}{b}\right)}\\
&&\hspace{-0.2cm}\label{cor3.5:r8}=\frac
{\left(\frac{qb}{dc},\frac{qb}{df},\boro{\frac{qb}{e}},\frac{d}{c},e;q\right)_n}
{\left(\frac{qb}{ce},\frac{qb}{cf},\frac{qb}{d},\frac{e}{d},d;q\right)_n}
{}_8W_7\left(\frac{q^{-n}d}{e};q^{-n},\frac{q^{-n}d}{b},
\frac{qb}{ec},\frac{qb}{ef},d;q,\frac{cf}{b}\right)\\
&&\hspace{-0.2cm}\label{cor3.5:r9}=
\boro{\frac{\left(\frac{qb}{ef},\frac{d}{c};q\right)_n}
{\left(\frac{qb}{ce},\frac{d}{f};q\right)_n}
{}_8W_7\left(\frac{q^{-n}f}{d};q^{-n},\frac{q^{-n}f}{b},
\frac{qb}{dc},
\frac{qb}{de},
f;q,\frac{ce}{b}\right)}\\
&&\hspace{-0.2cm}\label{cor3.5:r10}=
\frac{\left(\frac{qb}{de},\boro{\frac{qb}{dc},\frac{qb}{f},\frac{d}{c},f};q\right)_n}
{\left(\frac{qb}{ce},\boro{\frac{qb}{cf},\frac{qb}{d}},\frac{f}{d},\boro{d};q\right)_n}
{}_8W_7\left(\frac{q^{-n}d}{f};q^{-n},\frac{q^{-n}d}{b},
\frac{qb}{fc},
\frac{qb}{fe},
d;q,\frac{ce}{b}\right)\\
&&\hspace{-0.2cm}\label{cor3.5:r11}=
\boro{\frac{\left(\frac{qb}{ed},\frac{qb}{f},\frac{d}{c},f;q\right)_n}
{\left(\frac{qb}{cf},\frac{qb}{d},\frac{f}{e},d;q\right)_n}
{}_8W_7\left(\frac{q^{-n}e}{f};q^{-n},\frac{q^{-n}e}{b},
\frac{qb}{fc},\frac{qb}{fd},
e;q,\frac{cd}{b}\right)}\\
&&\hspace{-0.2cm}\label{cor3.5:r12}=
\boro{
\frac{\left(\frac{qb}{df},\frac{qb}{e},\frac{d}{c},e;q\right)_n}
{\left(\frac{qb}{ce},\frac{qb}{d},\frac{e}{f},d;q\right)_n}
{}_8W_7\left(\frac{q^{-n}f}{e};q^{-n},\frac{q^{-n}f}{b},
\frac{qb}{ec},\frac{qb}{ed},
f;q,\frac{cd}{b}\right)}
.
\end{eqnarray}
\end{cor}
}
\begin{proof}
\medskip
\noindent
Start with \eqref{cor3.5a.2} and consider all permutations of the symmetric parameters $c,d,e,f$ which produce non-trivial transformations. The ordering of the elements is given
by the first argument of the ${}_8W_7$ as follows: $
\{(c,d),(d,c),(c,e),(e,c),(c,f),(f,c),\ldots,(e,f),(f,e)\}.$
\end{proof}

\begin{cor}\label{cor:3.6}
Let $n\in\mathbb N_0$, $b,c,d,e,f,q\in\CCast$, 
such that $|q|\ne 1$.
Then, one has 
the following parameter interchange transformations for a terminating ${}_8W_7$:
\vspace{12pt}
\begin{eqnarray}
&&\hspace{-2.2cm}\label{cor3.6:r1}
{}_8W_7\left(\frac{qb^2}{def};q^{-n},\frac{qb}{de},
\frac{qb}{df},\frac{qb}{ef},c;q,\frac{q^{n+1}{b}}{c}\right)
\\
&&\hspace{-0.2cm}\label{cor3.6:r2}
=\frac{\left(\frac{qb}{c},\frac{q^2b^2}{def};q\right)_n}
{\left(\frac{qb}{d},\frac{q^2b^2}{cef};q\right)_n}
{}_8W_7\left(\frac{qb^2}{cef};q^{-n},\frac{qb}{ce},
\frac{qb}{cf},\frac{qb}{ef},d;q,\frac{q^{n+1}{b}}{d}\right)
\\
&&\hspace{-0.2cm}\label{cor3.6:r3}
=\frac{\left(\frac{qb}{c},\frac{q^2b^2}{def};q\right)_n}
{\left(\frac{qb}{e},\frac{q^2b^2}{cdf};q\right)_n}
{}_8W_7\left(\frac{qb^2}{cdf};q^{-n},\frac{qb}{cd},
\frac{qb}{cf},\frac{qb}{df},e;q,\frac{q^{n+1}{b}}{e}\right)
\\
&&\hspace{-0.2cm}\label{cor3.6:r4}
=\frac{\left(\frac{qb}{c},\frac{q^2b^2}{def};q\right)_n}
{\left(\frac{qb}{f},\frac{q^2b^2}{cde};q\right)_n}
{}_8W_7\left(\frac{qb^2}{cde};q^{-n},
\frac{qb}{cd},\frac{qb}{ce},\frac{qb}{de},f;q,\frac{q^{n+1}{b}}{f}\right).
\end{eqnarray}
\end{cor}

\begin{proof}
\noindent \boro{
Start with \eqref{cor3.5a.6} and consider all permutations of the symmetric parameters $c,d,e,f$ which produce non-trivial transformations.
}
\end{proof}

\zoro{
\begin{rem}
Other sets of parameter interchange transformations
can be obtained by considering all permutations
of the symmetric parameters $c,d,e,f$ in
\eqref{cor3.5a.7},
{
\eqref{cor3.5a.7b}, 
\eqref{cor3.5a.7c},
respectively}. However, one can see that
these are equivalent to the above Corollary 
\ref{cor:3.6} by replacing
\begin{eqnarray}
&&(b,c,d,e,f)\mapsto\left(\frac{q^{{-2n-1}}def}{b^2},\frac{q^{{-n-1}}cdef}{b^2},\frac{q^{-n}f}{b},\frac{q^{-n}e}{b},\frac{q^{-n}d}{b}\right)%
{,}\\
&&{(b,c,d,e,f)\mapsto\left(\frac{q^{{-2n-1}}cde}{b^2},\frac{q^{{-n-1}}cdef}{b^2},\frac{q^{-n}c}{b},\frac{q^{-n}d}{b},\frac{q^{-n}e}{b}\right)},\\
&&{(b,c,d,e,f)\mapsto\left(\frac{q^{-n-1}def}{b},\frac{q^{-n-1}cdef}{b^2},f,e,d\right),}
\end{eqnarray}
{respectively}.
\end{rem}
}


\boro{
\begin{cor}
\label{cor3.8}
Let $n\in\mathbb N_0$, $b,c,d,e,f,q\in\CCast$, 
such that $|q|\ne 1$.
Then, one has 
the following parameter interchange transformations
for a terminating ${}_4\phi_3$:
\begin{eqnarray}
&&\hspace{-2.2cm}\label{cor3.8:r1}\qhyp43{q^{-n},\frac{qb}{ef},c,d}
{\frac{q^{-n}cd}{b},\frac{qb}{e},\frac{qb}{f}}{q,q}\\
&&\hspace{-0.2cm}\label{cor3.8:r2}=
\frac{\left(\frac{qb}{de},\frac{qb}{c};q\right)_n}
{\left(\frac{qb}{cd},\frac{qb}{e};q\right)_n}\qhyp43{q^{-n},
\frac{qb}{cf},d,e}{\frac{q^{-n}de}{b},\frac{qb}{c},\frac{qb}{f}}{q,q}\\
&&\hspace{-0.2cm}\label{cor3.8:r3}=
\frac{\left(\frac{qb}{df},\frac{qb}{c};q\right)_n}
{\left(\frac{qb}{cd},\frac{qb}{f};q\right)_n}\qhyp43{q^{-n},\frac{qb}{ce},d,f}
{\frac{q^{-n}df}{b},\frac{qb}{c},\frac{qb}{e}}{q,q}\\
&&\hspace{-0.2cm}\label{cor3.8:r4}=
\frac{\left(\frac{qb}{ce},\frac{qb}{d};q\right)_n}
{\left(\frac{qb}{cd},\frac{qb}{e};q\right)_n}\qhyp43{q^{-n},
\frac{qb}{df},c,e}{\frac{q^{-n}ce}{b},\frac{qb}{d},\frac{qb}{f}}{q,q}\\
&&\hspace{-0.2cm}\label{cor3.8:r5}=
\frac{\left(\frac{qb}{cf},\frac{qb}{d};q\right)_n}
{\left(\frac{qb}{cd},\frac{qb}{f};q\right)_n}\qhyp43{q^{-n},\frac{qb}{de},c,f}
{\frac{q^{-n}cf}{b},\frac{qb}{d},\frac{qb}{e}}{q,q}\\
&&\boro{\hspace{-0.2cm}\label{cor3.8:r6}=
\frac{\left(\frac{qb}{ef},\boro{\frac{qb}{c}},\frac{qb}{d};q\right)_n}
{\left(\frac{qb}{de},\boro{\frac{qb}{e}},\frac{qb}{f};q\right)_n}\qhyp43{q^{-n},
\frac{qb}{cd},e,f}{\frac{q^{-n}ef}{b},\frac{qb}{c},\frac{qb}{d}}{q,q}.}
\end{eqnarray}
\end{cor}
}
\begin{proof}
\noindent \boro{
Start with \eqref{cor3.5a.3} and consider all permutations of the symmetric parameters $c,d,e,f$ which produce non-trivial transformations.
}
\end{proof}

\begin{rem}
Another set of parameter interchange transformations
can be obtained by considering all permutations
of the symmetric parameters $c,d,e,f$ in
\eqref{cor3.5a.4}. However, one can see that
these are equivalent to the above Corollary 
\ref{cor3.8} by replacing
\[
(b,c,d,e,f)\mapsto\left(\frac{q^{-2n}}{b},\frac{q^{-n}c}{b},\frac{q^{-n}d}{b},\frac{q^{-n}e}{b},\frac{q^{-n}f}{b}\right).
\]
\end{rem}

\boro{
\begin{cor}
\label{cor3.10}
Let $n\in\mathbb N_0$, $b,c,d,e,f,q\in\CCast$, 
such that $|q|\ne 1$.
Then, one has 
the following parameter interchange transformations for a terminating ${}_4\phi_3$:
\begin{eqnarray}
&&\hspace{-2.2cm}\label{cor3.10:r1}
\qhyp43{q^{-n},\frac{qb}{cd},\frac{qb}{ce},\frac{qb}{cf}}
{\frac{q^2b^2}{cdef},\frac{q^{1-n}}{c},\frac{qb}{c}}{q,q}\\
&&\hspace{-0.2cm}\label{cor3.10:r2}=
\frac{\left(\frac{qb}{d},d;q\right)_n}
{\left(\frac{qb}{c},c;q\right)_n}
\qhyp43{q^{-n},\frac{qb}{dc},\frac{qb}{de},\frac{qb}{df}}
{\frac{q^2b^2}{cdef},\frac{q^{1-n}}{d},\frac{qb}{d}}{q,q}\\
&&\hspace{-0.2cm}\label{cor3.10:r3}=
\frac{\left(\frac{qb}{e},e;q\right)_n}
{\left(\frac{qb}{c},c;q\right)_n}
\qhyp43{q^{-n},\frac{qb}{ec},\frac{qb}{ed},\frac{qb}{ef}}
{\frac{q^2b^2}{cdef},\frac{q^{1-n}}{e},\frac{qb}{e}}{q,q}\\
&&\hspace{-0.2cm}\label{cor3.10:r4}=
\frac{\left(\frac{qb}{f},f;q\right)_n}
{\left(\frac{qb}{c},c;q\right)_n}
\qhyp43{q^{-n},\frac{qb}{fc},\frac{qb}{fd},\frac{qb}{fe}}
{\frac{q^2b^2}{cdef},\frac{q^{1-n}}{f},\frac{qb}{f}}{q,q}.
\end{eqnarray}
\end{cor}
}
\begin{proof}
\noindent \boro{
Start with \eqref{cor3.5a.5} and consider all permutations of the symmetric parameters $c,d,e,f$ which produce non-trivial transformations.
}
\end{proof}

\begin{rem}
Another set of parameter interchange transformations
can be obtained by considering all permutations
of the symmetric parameters $c,d,e,f$ in
\eqref{cor3.5a.6b}. However, one can see that
these are equivalent to the above Corollary 
\ref{cor3.10} by replacing
\[
(b,c,d,e,f)\mapsto\left(\frac{q^{-n}f}{e},\frac{qb}{ce},\frac{qb}{de},f,\frac{q^{-n}f}{b}\right).
\]
\end{rem}


The $q$-inverse Askey--Wilson polynomials are simply a scaled version of the Askey--Wilson 
polynomials with the free parameters $a_k$ replaced by their reciprocals $a_{k}^{-1}$.
We demonstrate this in the following remark.

\begin{rem}
\label{invAWrem}
Let ${\mathsf p}_n(\theta;a_1,a_2,a_3,a_4|q):=p_n(x;{\bf a}|q)$, where $x=\cos\theta$,
be any representation of the Askey--Wilson polynomials. Then the
$q$-inverse Askey--Wilson polynomials
$p_n(x;{\bf a}|q^{-1})$ are
given by
\begin{eqnarray*}
&&\hspace{-1cm}{\mathsf p}_n(\theta;a_1,a_2,a_3,a_4|q^{-1})
=q^{-3\binom{n}{2}}(-a_{1234})^n
{\mathsf p}_n\!\left(\!\left.-\theta;
a_1^{-1}, a_2^{-1}, a_3^{-1}, a_4^{-1}
\right|q\right)\\
&&\hspace{-1cm}\hspace{3.46cm}=q^{-3\binom{n}{2}}(-a_{1234})^n
{\mathsf p}_n\!\left(\left.\theta;
a_1^{-1}, a_2^{-1}, a_3^{-1}, a_4^{-1}
\right|q\right),
\end{eqnarray*}
where the second equality 
follows from Remark
\ref{rem:2.7}. So aside from a specific normalization, as is well-known, 
the $q$-inverse Askey--Wilson polynomials are the Askey--Wilson polynomials 
with the parameters taken to be their reciprocals. 
Note that this will not be the case for the symmetric subfamilies of 
the Askey--Wilson polynomials.
\end{rem}

Nonetheless, we give in the following corollary the 
terminating basic hypergeometric representations of these polynomials.
\begin{cor} \label{cor:3.13}
Let $p_n(x,{\bf a}|q)$ and all the respective parameters be defined as previously. 
Then, the $q$-inverse Askey--Wilson polynomials are given by:
\begin{eqnarray}
\label{qiaw:1}\hspace{-0.55cm}
p_n(x;{\bf a}|q^{-1}) &&=
q^{-3\binom{n}{2}} (-a_pa_{1234})^n \left(\left\{\frac{1}
{a_{ps}}\right\}_{s \neq p};q\right)_n \qhyp43{q^{-n},\frac{q^{n-1}}{a_{1234}}, 
\frac{\expe^{\pm i\theta}}{a_p}}{\left\{\frac{1}{a_{ps}}\right\}_{s \neq p}}{q,q} \\
\label{qiaw:2}&&=q^{-4\binom{n}{2}}(a_p a_{1234})^n 
\left(\frac{q^{n-1}}{a_{1234}}, \frac{\expe^{\pm i\theta}}{a_p};q\right)_n 
\qhyp43{q^{-n}, \{q^{1-n}a_{ps}\}_{s \neq p}}
{q^{2-2n}a_{1234} ,q^{1-n}a_p\expe^{\pm i\theta}}{q,q} \\
\label{qiaw:3} &&=q^{-3\binom{n}{2}}
\bigl(-a_{1234}\expe^{-i\theta}\bigr)^n \left(\frac{1}{a_{pr}},\frac{\expe^{i\theta}}{a_t}, 
\frac{\expe^{i\theta}}{a_u};q\right)_n \qhyp43{q^{-n}, \frac{\expe^{-i\theta}}{a_p}, 
\frac{\expe^{-i\theta}}{a_r}, q^{1-n}a_{tu}}
{\frac{1}{a_{pr}}, 
q^{1-n}a_t \expe^{-i\theta}, q^{1-n}a_u \expe^{-i\theta}}{q,q} \\
&&=q^{-3{n\choose 2}}
(-a_{1234} \expe^{-i\theta})^n 
\dfrac{\left(\frac{1}{q a_{1234}};q\right)_{2n}\left(\left\{\frac{\expe^{i\theta}}
{a_s}\right\}_{s\ne p}
,\frac{a_p\,\expe^{i\theta}}{q a_{1234}};q\right)_{n}}
{\left(\frac{1}{qa_{1234}};q\right)_{n}
\left(\frac{a_p\,\expe^{i\theta}}{q a_{1234}};q\right)_{2n}}
\nonumber \\
&&\label{qiaw:6}
\boro{\hspace{2.5cm}\times\,{}_8W_7\left(
\frac{q^{1-2n}a_{1234}\expe^{-i\theta}}{a_p};q^{-n},
\left\{\frac{q^{1-n}a_{1234}}{a_{ps}}\right\}_{s\ne p},
\frac{\expe^{-i\theta}}{a_{p}};q,
q a_p\,\expe^{-i\theta}\right)}
\\
&&\label{qiaw:7}
\boro{=
q^{-3{n\choose 2}}
(-a_{1234} \expe^{-i\theta})^n 
\dfrac{\left(\frac{\expe^{i\theta}}{a_p},\{\frac{a_{ps}}
{a_{1234}}\}_{s\ne p};q\right)_n}
{\left(\frac{a_p\,\expe^{-i\theta}}{a_{1234}};q\right)_{n}}}
\nonumber \\
&&\boro{\hspace{2.5cm}\times\,{}_8W_7\left(
\frac{a_p\expe^{-i\theta}}{qa_{1234}};q^{-n},
\left\{\frac{\expe^{-i\theta}}{a_s}\right\}_{s\ne p}, \frac{q^{n-1}}{a_{1234}};q,
qa_p\,\expe^{i\theta}\right)}\\
&&=q^{-3\binom{n}{2}} 
\left(-a_pa_{1234}\right)^n
\frac{\left(
\frac{1}{a_{pt}},\frac{1}{a_{pu}},
\frac{\expe^{\pm i\theta}}{a_r}
;q\right)_n }
{(\frac{a_p}{a_r};q)_n}
\nonumber \\
\label{qiaw:5} 
&& \hspace{2.2 cm} \times
{}_8W_7\left(
\frac{q^{-n}a_r}{a_p};q^{-n},
q^{1-n}a_{rt}, 
q^{1-n}a_{ru}, 
\frac{\expe^{\pm i\theta}}{a_p};
q,\frac{q^n}{a_{tu}}
\right)\\
&&=q^{-3\binom{n}{2}} \bigl(-a_{1234}\expe^{-i\theta}\bigr)^n 
\frac{(\{\frac{\expe^{i\theta}}{a_s}\};q)_n}{(\expe^{2i\theta};q)_n} 
\label{qiaw:4} 
{}_8W_7\left(
q^{-n}\expe^{-2i\theta};q^{-n},
\left\{\frac{\expe^{-i\theta}}{a_s}\right\};q
,q^{2-n}a_{1234}
\right).
\end{eqnarray}
\end{cor}
\begin{proof}
Applying Theorem 
\ref{thm:2.6} to the terminating basic hypergeometric representations
of the Askey--Wilson polynomials \eqref{aw:def1}--\eqref{aw:def4}
produces the inverted basic hypergeometric representations
\eqref{qiaw:1}--\eqref{qiaw:4}.
This completes the proof.
\end{proof}
\vspace{6pt} 
\section*{Acknowledgments} H.S.C. would like to thank Mourad Ismail for valuable discussions..
R.S.C-S acknowledges financial support through the 
research project PGC2018–096504-B-C33 supported by 
Agencia Estatal de Investigaci\'on of Spain.


\begin{thebibliography}{999}

\bibitem
{Koekoeketal}
{Koekoek,} {R.}; Lesky, P.A.; Swarttouw, R.F.
\newblock {\em Hypergeometric Orthogonal Polynomials and Their
{$q$}-Analogues}; Springer Monographs in Mathematics; Springer-Verlag:
Berlin, Germany, 2010 

\bibitem
{Rosengren2007}
Rosengren, H.
\newblock An elementary approach to {$6j$}-symbols (classical, quantum,
rational, trigonometric, and elliptic).
\newblock {\em Ramanujan J.} {\bf 2007}, {\em 13}, 131--166.
\newblock
doi:\href{https://doi.org/10.1007/s11139-006-0245-1}{\detokenize{10.1007/s11139-006-0245-1}}. [\href{http://dx.doi.org/10.1007/s11139-006-0245-1}{CrossRef}]

\bibitem
{IsmailZhang2019}
Ismail, M.E.H.; Zhang, R.
\newblock New Orthogonal Polynomials of Askey-Wilson Type.
\newblock  {\bf 2020}, in preparation.

\bibitem
{VanderJeugtRao1999}
Van der Jeugt, J.; Rao, S.K.
\newblock Invariance groups of transformations of basic hypergeometric series.
\newblock {\em \mbox{J. Math. Phys.}} {\bf 1999}, {\em
40}, 6692--6700.
\newblock
doi:{\href{https://doi.org/10.1063/1.533115}{\detokenize{10.1063/1.533115}}}. [\href{http://dx.doi.org/10.1063/1.533115}{CrossRef}]

\bibitem
{KrattenthalerRao2004}
Krattenthaler, C.; Rao, S.K.
\newblock On group theoretical aspects, hypergeometric transformations and
symmetries of angular momentum coefficients. In {\em Symmetries in Science
{XI}}; Kluwer Academic Publishers: Dordrecht,  The Netherlands, 2004; pp. 355--376.

\bibitem
{Kajihara2014}
Kajihara, Y.
\newblock Symmetry groups of {$A_n$} hypergeometric series.
\newblock {\em SIGMA. Symmetry Integr. Geom.} {\bf 2014}, {\em 10}, 026.
\newblock
doi:{\href{https://doi.org/10.3842/SIGMA.2014.026}{\detokenize{10.3842/SIGMA.2014.026}}}. [\href{http://dx.doi.org/10.3842/SIGMA.2014.026}{CrossRef}]

\bibitem
{GaspRah}
Gasper, G.; Rahman, M.
\newblock {\em Basic Hypergeometric Series}, 2nd ed.; {Encyclopedia of Mathematics and Its Applications}; Cambridge University
Press: Cambridge, UK, 2004; Volume 96. 

\bibitem
{NIST:DLMF}
Olver, F.W.J.; Olde Daalhuis, A.B.; Lozier, D.W.; Schneider, B.I.; Boisvert, R.F.; Clark, C.W.; Miller, B.R.; Saunders, B.V.; Cohl, H.S.; McClain, M.A. (Eds). NIST Digital Library of Mathematical Functions.
\newblock  Available online: \url{http://dlmf.nist.gov/},
Release 1.0.28 of 2020-09-15
.(
Accessed on 15 September 2020
).

\bibitem
{KoornwinderKLSadd}
{Koornwinder}, T.H.
\newblock {Additions to the formula lists in ``Hypergeometric orthogonal
polynomials and their $q$-analogues'' by Koekoek, Lesky and Swarttouw}.
\newblock {\bf 2015}, arXiv:1401.0815v2.

\bibitem
{IsmailWilson82}
Ismail, M.E.H.; Wilson, J.A.
\newblock Asymptotic and generating relations for the {$q$}-{J}acobi and
{$_{4}\varphi _{3}$} polynomials.
\newblock {\em J. Approx. Theory} {\bf 1982}, {\em 36}, 43--54.
\newblock
doi:{\href{https://doi.org/10.1016/0021-9045(82)90069-7}
{\detokenize{10.1016/0021-9045(82)90069-7}}}. [\href{http://dx.doi.org/10.1016/0021-9045(82)90069-7}{CrossRef}]

\bibitem
{AskeyWilson85}
Askey, R.; Wilson, J.
\newblock Some basic hypergeometric orthogonal polynomials that generalize
{J}acobi polynomials.
\newblock {\em Mem. Am. Math. Soc.} {\bf 1985}, {\em
54}.
\newblock doi:{\href{https://doi.org/10.1090/memo/0319}{\detokenize{10.1090/memo/0319}}}. 
[\href{http://dx.doi.org/10.1090/memo/0319}{CrossRef}]

\end{thebibliography}
\end{document}